\documentclass[a4paper]{amsart}

\usepackage{amsmath}
\usepackage{amssymb}
\usepackage{amsthm}
\usepackage{graphicx}
\usepackage{bbm}

\newtheorem{theorem}{Theorem}

\newtheorem{lemma}[theorem]{Lemma}

\title{Disproof of a conjecture by Rademacher on partial fractions}

\author{Michael Drmota}
\address{Vienna University of Technology, Wiedner Hauptstra\ss{}e 8--10/105-1,
A-1040 Vienna, Austria}
\email{michael.drmota@tuwien.ac.at}

\author{Stefan Gerhold}
\address{Vienna University of Technology, Wiedner Hauptstra\ss{}e 8--10/105-1,
A-1040 Vienna, Austria}
\email{sgerhold@fam.tuwien.ac.at}
\thanks{S.~Gerhold gratefully acknowledges financial support from the
Austrian Science Fund (FWF) under grant P~24880-N25.}

\date{\today}

\begin{document}

\begin{abstract}
  In his book Topics in Analytic Number Theory,
  Rademacher considered the generating function
  of partitions into at most $N$ parts, and conjectured certain limits
for the coefficients of its partial fraction decomposition. We carry out an asymptotic
analysis that disproves this conjecture, thus confirming recent observations
of Sills and Zeilberger (Journal of Difference Equations and Applications 19, 2013),
who gave strong  numerical evidence against the conjecture.
\end{abstract}



\maketitle

\section{introduction}

In his book Topics in Analytic Number Theory~\cite{Ra73},
Rademacher gave a partial fraction decomposition of the
partition generating function $\prod_{j\geq 1}(1-x^j)^{-1}$.
He conjectured that the decomposition
of the generating function of partitions into at most~$N$
parts,
\[
  \prod_{j=1}^N \frac{1}{1-x^j} = \sum_{k=1}^N
    \sum_{\substack{0\leq h<k \\ \gcd(h,k)=1}} \sum_{l=1}^{\lfloor N/k \rfloor}
    \frac{C_{h,k,l}(N)}{(x-e^{2\pi i h/k})^l},
\]
is consistent with it in the sense that the coefficients~$C_{h,k,l}(N)$
converge as $N\to\infty$ to the coefficients of the decomposition
of the unrestricted generating function.
Despite attracting the attention of several authors~\cite{An07,Eh93,Mu08},
the conjecture has been open since 1973.
See Sills and Zeilberger~\cite{SiZe13} for some further historical remarks.
The latter paper presents a recurrence
for $C_{0,1,l}(N)$; the values computed by it
do not seem to show convergence, but 
rather  oscillating and unbounded behavior.
It is well known, though, that there are number-theoretic problems
where the true asymptotics are numerically visible only
for \emph{very} large values. See, e.g.,~\cite{Ge10a} for an example.
The present note rigorously confirms the main observation
from~\cite{SiZe13}, i.e., we disprove Rademacher's
conjecture.\footnote{It is important to note that Cormac O'Sullivan
has disproved Rademacher's conjecture independently from us,
as announced  in his paper~\cite{oS12}, with a different approach.
More precisely, he proved that there exist $h,k \le 100$ such
that $C_{h,k,1}(N)$  does not converge (personal communication), whereas our method
proves directly a conjectural relation from~\cite{oS12} (Conjecture 6.2)
-- with a slightly worse error term.
}
To formulate our main result, recall the definition of the dilogarithm function: $\mathrm{Li}_2(w)
= \sum_{k\geq1}w^k/k^2$, $|w|<1$.  Define $z_0\approx -1.61 + 7.42i$
as the solution of
\begin{equation}\label{eq:z0}
  \log(1-e^z) + \frac{1}{z}(\mathrm{Li}_2(e^z)-\pi^2/6) = 0.
\end{equation}
(It is easy to show that there is a unique root within, say, distance
$1$ of the numerical value given above.)
Furthermore, define $\rho=\exp(i a)$, where
\begin{equation}\label{eq:a}
  a = \frac{ \pi}{2} - \frac12 \arg\left( \frac{e^{z_0}}{z_0(1-e^{z_0})} \right) \approx 1.79.
\end{equation}
\begin{theorem}\label{thm:main}
    For any integer $l\geq1$, we have the asymptotics
    \begin{equation}\label{eq:C asympt}
        C_{0,1,l}(N) = b^N N^{-l-1} H_l(N) + O(b^N N^{-l-117/112}), \quad N\to\infty, 
    \end{equation}
    where $b=1/|1-e^{z_0}|\approx 1.07$, and~$H_l$ is a bounded function with period
   $p=2\pi/|\arg(1-e^{z_0})|\approx 31.96$, given by
   \begin{multline*}
       H_l(N) = \frac{(-1)^{l-1}}{\pi} 
       \sqrt{-\frac {z_0(1-e^{z_0})}{\rho^2e^{z_0}}}
       \Bigg(\Im\left( \frac{\rho(-z_0)^{l-1/2}}{\sqrt{1-e^{z_0}}}\right)\cos\big(N \arg(1-e^{z_0})\big) \\
        -\Re\left(\frac{\rho(-z_0)^{l-1/2}}{\sqrt{1-e^{z_0}}}\right)\sin\big(N \arg(1-e^{z_0})\big)\Bigg).
   \end{multline*}
\end{theorem}
Note that the number under the first radical sign is real and positive.
The period~$p$
of the oscillations is roughly~$32$, as observed by Sills and
Zeilberger~\cite{SiZe13}. It is independent of~$l$, as is the
exponential growth order $b^N$. Moreover, Sills and Zeilberger
found that the successive peaks seem to grow exponentially with
a factor around~$8$. The (asymptotically) true factor
is $b^p\approx 8.81$.
Figures~\ref{fig:1} and~\ref{fig:2} illustrate the quality of the approximation,
which seems to be better for $l=1$ than for $l=2$.
\begin{figure}
\begin{center}
\includegraphics[height=2in]{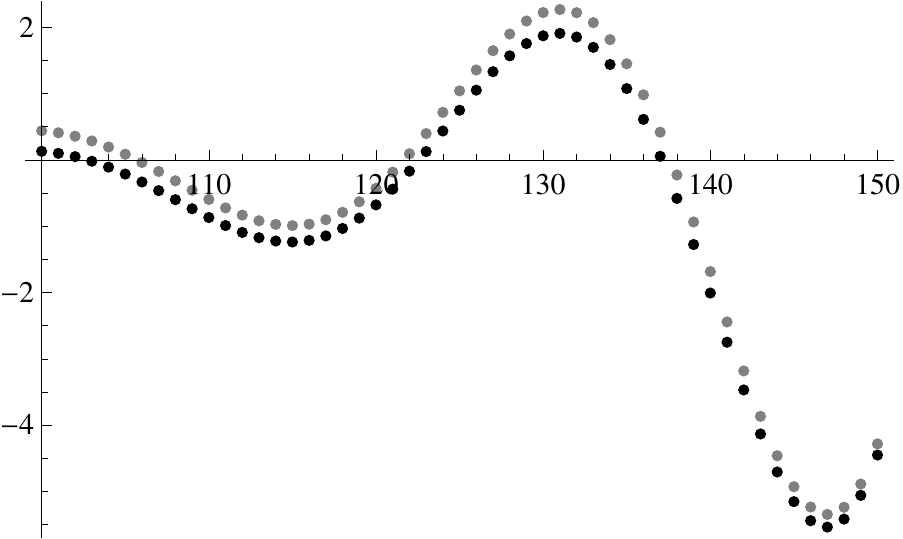}
\end{center}
\caption{\label{fig:1}The numbers $C_{0,1,1}(N)$ (black)
and the approximation~\eqref{eq:C asympt} (gray), for
$N=100,\dots,150$.}
\end{figure}
\begin{figure}
\begin{center}
\includegraphics[height=2in]{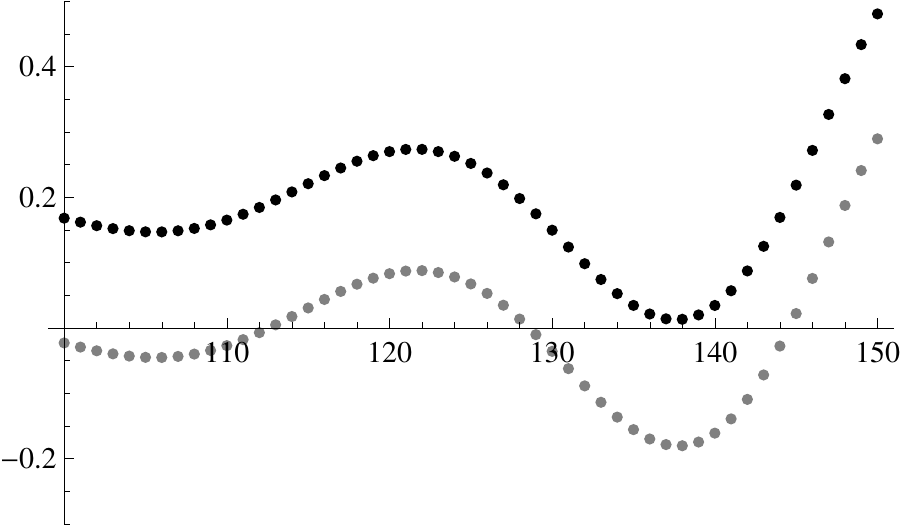}
\end{center}
\caption{\label{fig:2}The numbers $C_{0,1,2}(N)$ (black)
and the approximation~\eqref{eq:C asympt} (gray), for
$N=100,\dots,150$.}
\end{figure}

In principle, it should be possible to extend our approach
from $(h,k)=(0,1)$ to general $h,k$.
Moreover, a natural conjecture is that the period $2\pi/|\arg(1-e^{z_0})|$
of~$H_l$ is a transcendental number.
While there is some literature on transcendence of polylogarithm values
(see, e.g.,~\cite{Ha93}), we are not aware of any result that would imply this.

The rest of the paper is organized as follows. In Section~\ref{se:mellin},
we appeal to the Cauchy integral representation of $C_{0,1,l}(N)$
and find an asymptotic approximation for its integrand.
The new integrand is analysed in Section~\ref{se:saddle}
by the saddle point method. Section~\ref{se:right} completes the proof
of Theorem~\ref{thm:main} by adding estimates in regions
where the asymptotic approximation for the integrand has to be
modified or is invalid. In the conclusion, we comment on the error term
in~\eqref{eq:C asympt}, and on possible future work.

\section{Mellin Transform asymptotics}\label{se:mellin}

Since the $C_{0,1,l}(N)$ are the Laurent coefficients of
$\prod_{j=1}^N (1-x^j)^{-1}$ at $x=1$, we
can express them by Cauchy's formula:
\begin{align}
   C_{0,1,l}(N) &= \frac{1}{2i\pi} \oint x^{l-1} \prod_{j=1}^N \frac{1}{1-(x+1)^j}dx \notag \\
   & =  \frac{(-1)^{l-1}}{N^l} \frac{1}{2i\pi} \oint F(z,N) dz, \label{eq:F}
\end{align}
where we have substituted $x+1=e^{z/N}$, and
\begin{equation}\label{eq:F def}
  F(z,N) := e^{f(z,N)} := e^{z/N}N^{l-1}(1-e^{z/N})^{l-1}\prod_{j=1}^N \frac{1}{1-e^{zj/N}}.
\end{equation}
We wish to replace the integrand~$F$ by an asymptotic approximation, 
derived by a Mellin transform approach.
We do the analysis for $\Re z<0$, since the factor~$e^{-Nz/2}$ in
\begin{equation}\label{eq:refl}
  \prod_{j=1}^N \frac 1{1-e^{zj/N}} 
  = (-1)^N e^{-z(N+1)/2} \prod_{j=1}^N \frac 1{1-e^{-zj/N}}
\end{equation}
suggests that the contribution of the left half-circle dominates
the integral~\eqref{eq:F}; a rigorous argument is given in Section~\ref{se:right}.
To take the Mellin transform of $f=\log F$ w.r.t.~$N$, we have to interpolate between integral values of~$N$. We therefore rewrite the
logarithm of the product $\prod_{j=1}^N$ in~\eqref{eq:F def} as follows:
\begin{align*}
 g(z,N)  &:= \log \prod_{j=1}^n  \frac{1}{1-e^{zj/N}} \\ 
 &= - \sum_{j=1}^N \log(1-e^{zj/N}) \\
  &= \sum_{j=1}^N \sum_{k=1}^\infty \frac{1}{k}e^{zjk/N} 
 = \sum_{k=1}^\infty \frac{1}{k} \frac{1-e^{kz}}{e^{-kz/N}-1}.
\end{align*}
Now we can compute the Mellin transform of~$g$ w.r.t.~$N$, for $\Re(s)<-1$:
\begin{align}
 \mathcal{M}g(z,\cdot)(s)&=\int_0^\infty g(z,x)x^{s-1}dx \notag \\
&=  \sum_{k=1}^\infty \frac{1-e^{kz}}{k}
    \int_0^\infty \frac{x^{s-1}}{e^{-kz/x}-1} dx \notag \\
  &= \sum_{k=1}^\infty  \frac{1-e^{kz}}{k} (-kz)^s \Gamma(-s)\zeta(-s) \notag \\
  &= (-z)^s \Gamma(-s)\zeta(-s) \left(\sum_{k=1}^\infty k^{s-1}
  - \sum_{k=1}^\infty  k^{s-1} e^{kz}  \right) \notag \\
  &= (-z)^s \Gamma(-s)\zeta(-s) \Big(\zeta(1-s) - \mathrm{Li}_{1-s}(e^z)\Big).
  \label{eq:mellin}
\end{align}
Recall that the polylogarithm is
defined for $|w|<1$  and $\nu\in\mathbb{C}$ by $\mathrm{Li}_\nu(w) = \sum_{k\geq1}w^k/k^\nu$.
For the integral evaluation used in the third equality,
see Titchmarsh~\cite{Ti86}, p.~18; it already appears in Riemann's
original memoir~\cite{Ri59}.
By the Mellin inversion formula~\cite{FlGoDu95},  $g$ can be recovered by
\begin{equation}\label{eq:g mell}
  g(z,N) = \frac{1}{2i\pi} \int_{-3/2-i\infty}^{-3/2+i \infty} \mathcal{M}g(z,\cdot)(s) N^{-s}ds.
\end{equation}
We now move the integration line to the right and collect residues.
To estimate the resulting integral (and justify Mellin inversion),
we first establish a bound on
$ \mathrm{Li}_{1-s}(e^z)$ for $|\Im s|$ large. Note that Pickard~\cite{Pi68}
studied asymptotics of $\mathrm{Li}_\nu(w)$ for $\nu\to0$ and $\nu\to\infty$,
and wrote that ``little is known about behavior in the $\nu$-plane
except along and near the line $(0,\infty)$.''
\begin{lemma}\label{le:Li}
  Suppose that $z$ is bounded, bounded away from~$0$ and $\pm 2i\pi$,
  $|\Im z|<8$, and $\Re z \leq 0$. Then,
  for $\Re s>1$ fixed and $|\Im s|\to\infty$, we have
  \[
     \mathrm{Li}_{1-s}(e^z) = O( |\Im s|^{\Re s -1/2}).
  \]
\end{lemma}
\begin{proof}
  We use the representation
  \begin{equation}\label{eq:Li Hurwitz}
    \mathrm{Li}_{1-s}(e^z) = \frac{\Gamma(s)}{(2\pi)^s}
      \left( i^s \zeta\left(s,\frac12 + \frac{\log(-e^z)}{2i\pi}\right)
      + i^{-s} \zeta\left(s,\frac12 - \frac{\log(-e^z)}{2i\pi}\right)\right),
  \end{equation}
  due to Jonqui{\`e}re~\cite{Jo89}, where
  \[
    \zeta(s,q) = \sum_{k=0}^\infty \frac{1}{(q+n)^s}, \quad \Re s >1, \Re q >0,
  \]
  is the Hurwitz zeta function. 
  First we establish some simple estimates for this function.
  Suppose that $\Im s \to +\infty$ and that $\Im q <0$. Since
  \begin{equation}\label{eq:q+n}
    |q+n|^{-s} = |q+n|^{-\Re s} e^{\Im(s) \arg(q+n)}
  \end{equation}
  and $\arg(q+n)<0$, we obtain
  \[
     |\zeta(s,q)| \leq \sum_{k=0}^\infty |q+n|^{-\Re s} = O(1),
  \]
  for bounded $q$ with $\Re q>0$ and $q$ bounded away from zero.
  If $\Im q >0$, on the other hand, we use the bound $e^{\Im(s) \arg(q+n)}
  \leq e^{\tfrac12 \pi \Im s}$ in~\eqref{eq:q+n} to conclude
  \[
     \zeta(s,q) = O(e^{\tfrac12 \pi \Im s}).
  \]
  Analogous bounds hold for $\Im s\to-\infty$. To apply them to~\eqref{eq:Li Hurwitz},
  note that
  \[
    \Re\left(\frac12 \pm \frac{\log(-e^z)}{2i\pi}\right)>0 \quad \text{and} \quad
    \Im\left(\frac12 + \frac{\log(-e^z)}{2i\pi}\right)>0
  \]
  in the specified range of~$z$.
  For the desired estimate, it now suffices to observe that
   $|i^{\pm s}|=\exp(\mp \tfrac12\pi
    \Im s)$, that the factor $(2\pi)^{-s}$ is $O(1)$, and that we have
  \[
    \Gamma(s) = O(e^{-\tfrac12 \pi |\Im s|} |\Im s|^{\Re s - 1/2})
  \]
  by Stirling's formula.
\end{proof}
We can now find the asymptotics of~$g$ by shifting the integration
in~\eqref{eq:g mell} to the right, where $\Re s =8/7$ turns out to
be a suitable choice.
The polylogarithm $\mathrm{Li}_{1-s}(e^z)$ is an entire function of~$s$.
Moreover, $\zeta(-s)$ has a simple pole at $s=-1$, and $\Gamma(-s)$ has simple
poles at the non-negative integers. Because of the factor $\zeta(1-s)$, the
transform~\eqref{eq:mellin} has a double pole at  $s=0$, which results in
a logarithmic term in the asymptotics of~$g$.
\begin{lemma}\label{le:mellin}
  For $\Re z<0$, the function $f$ defined in~\eqref{eq:F def} has the representation
  \begin{multline}\label{eq:f asympt}
     f(z,N) = \frac{1}{z}\Big( \mathrm{Li}_2(e^z) - \frac{\pi^2}{6}\Big)N
   - \frac12 \log N  \\
    - \frac12 \Big(\log 2\pi    +\log(1-e^{z})
 - \log(-z)\Big) +(l-1)\log(-z)+ h(z,N),
 \end{multline}
  where~$h$ is given by
  \begin{multline}\label{eq:h}
    h(z,N) =  \frac{z(e^z+1)}{24(e^z-1)} \frac1N +\frac{z}{N}
      + \frac{1}{2i\pi}\int_{8/7-i\infty}^{8/7+i\infty}  \mathcal{M}g(z,\cdot)(s) N^{-s} ds \\
                -(l-1) \log(-z) + (l-1) (\log N + \log (1-e^{z/N})). 
  \end{multline}
  The function~$h$ is
  \begin{itemize}
    \item[(i)] uniformly $O(N^{-1/2})$ if $|\arg z|\geq \pi/2+\varepsilon$, $z$ is  bounded away
      from~$0$, and $z=O(N^{1/2})$,
    \item[(ii)] uniformly $O(N^{33/112})$ if $z$ is bounded, bounded away from~$0$ and $\pm 2i\pi$,
      $|\Im z|<8$, and $\Re z < -N^{-7/8}$.
  \end{itemize}
\end{lemma}
\begin{proof}
    We shift the integration in~\eqref{eq:g mell} to $\Re s=8/7$.
   The residues of~\eqref{eq:mellin} at $s=-1$, $s=0$, and $s=1$ are straightforward
   to compute and yield
  \begin{multline*}
         g(z,N) = \frac{1}{z}\Big( \mathrm{Li}_2(e^z) - \frac{\pi^2}{6}\Big)N
   - \frac12 \log N 
    - \frac12 \Big(\log 2\pi    +\log(1-e^{z})  - \log(-z)\Big)\\
         + \frac{z(e^z+1)}{24(e^z-1)} \frac1N
         +\frac{1}{2i\pi}\int_{8/7-i\infty}^{8/7+i\infty}  \mathcal{M}g(z,\cdot)(s) N^{-s} ds.
  \end{multline*}
   Together with the definition of $f=\log F$ in~\eqref{eq:F}, we obtain~\eqref{eq:f asympt}.
   Except for the integral, it is immediate that all terms in~\eqref{eq:h} satisfy
   the bounds stated in~(i) and~(ii). Note  that
   \[
     \log(1-e^{z/N}) = \log(-z) - \log N + O(N^{-1/2})
   \]
   in both cases (i) and (ii), and that
   \[
    \left |\frac{e^z+1}{e^z-1}\right| 
      = \frac{1+e^{2 \Re z} +2e^{ \Re z} \cos(\Im z)}
        {1+e^{2 \Re z} -2e^{ \Re z} \cos(\Im z)}
      \leq \frac{1+e^{2 \Re z} +2e^{ \Re z} }
        {1+e^{2 \Re z} -2e^{ \Re z} }
  \]
  is bounded. To estimate the integral in~\eqref{eq:h}, we use
  the following well-known equations resp.\ estimates, for $\Re s= 8/7$ and
  $\Im s\to \infty$ ($\Im s<0$ is treated by conjugation):
  \begin{align}
    |N^{-s}| &= N^{-\Re s}, \notag \\
    |(-z)^s| &= |z|^{\Re s} e^{-\Im(s) \arg(-z)}, \label{eq:-z} \\
    |\Gamma(-s)| &\sim \sqrt{2\pi} e^{-\tfrac12\pi \Im s} (\Im s)^{-\Re s - 1/2},  \label{eq:stirling}  \\
    \zeta(-s) &= O((\Im s)^{\Re s+1/2}),  \label{eq:zeta} \\
    \zeta(1-s)&= O((\Im s)^{ \Re s-1/2}).   \label{eq:zeta1}
  \end{align}
  For~\eqref{eq:zeta} and~\eqref{eq:zeta1}, see Titchmarsh~\cite{Ti86}, p.~95.
  In case~(i), we have
  \[
    |\mathrm{Li}_{1-s}(e^z) | \leq \mathrm{Li}_{\Re(1-s)}(e^{\Re z}) = O(1)
  \]
  by the triangle inequality and the analyticity of the polylogarithm
  in the unit disk.
  Since $|\arg z|\geq \pi/2+\varepsilon$, we see from~\eqref{eq:-z}
  and~\eqref{eq:stirling} that the integrand has an exponentially decaying factor
  $\exp(-\Im(s) (\arg(-z) + \tfrac12 \pi )) \leq \exp(-\varepsilon\Im s)$.
  The integral is thus $O(N^{-\Re s}|z|^{\Re s})=O(N^{-4/7})=O(N^{-1/2})$.
  
  In case~(ii), the decay of the exponential bound slows down as~$N$ increases,
  because $\arg(-z)$ may approach $-\pi/2$, and we must also take into
  account the powers of~$\Im s$ in the estimates \eqref{eq:-z}--\eqref{eq:zeta1}
  and Lemma~\ref{le:Li}.
  The boundedness of~$z$ guarantees 
  that $|z|^{\Re s}$ in~\eqref{eq:-z} stays bounded,
  and that $\arg(-z) + \pi/2 \sim -\Re z$ for $\Re z\to0$.
  We can thus bound the integral by a constant multiple of
  \[
    N^{-\Re s}\int_0^\infty e^{-N^{-7/8}x} x^{\Re s-1/2}dx=
     N^{-8/7} N^{23/26}=O(N^{33/112}).
  \]
  Note that the powers of $\Im s$ in~\eqref{eq:stirling} and~\eqref{eq:zeta} cancel, and
  that~\eqref{eq:zeta1} and Lemma~\ref{le:Li} show that the term in parentheses
  in~\eqref{eq:mellin} is $O((\Im s)^{\Re s-1/2})$.
\end{proof}

Lemma~\ref{le:mellin} suggests the approximate integral representation
\begin{multline}\label{eq:new int}
  C_{0,1,l}(N) \\ \approx \frac{(-1)^{l-1}}{N^{l+1/2}(2\pi)^{3/2}i}
  \int_{|z|=5, \Re z\leq 0}
  \frac{(-z)^{l-1/2}}{\sqrt{1-e^z}} \exp\left( \frac{z}{N} + \frac{N}{z}
  \left( \mathrm{Li}_2(e^z) - \frac{\pi^2}{6}\right) \right) dz,
\end{multline}
where~$h$ from~\eqref{eq:f asympt} has been replaced by zero,
except the term $z/N$, which was retained for better accuracy.
Recall that the right half-circle is negligible, as mentioned
above and proved in Section~\eqref{se:right}.
Even for small~$N$, the fit is very good for $l=1$; see Figure~\ref{fig:approx}.
\begin{figure}
\begin{center}
\includegraphics[height=2in]{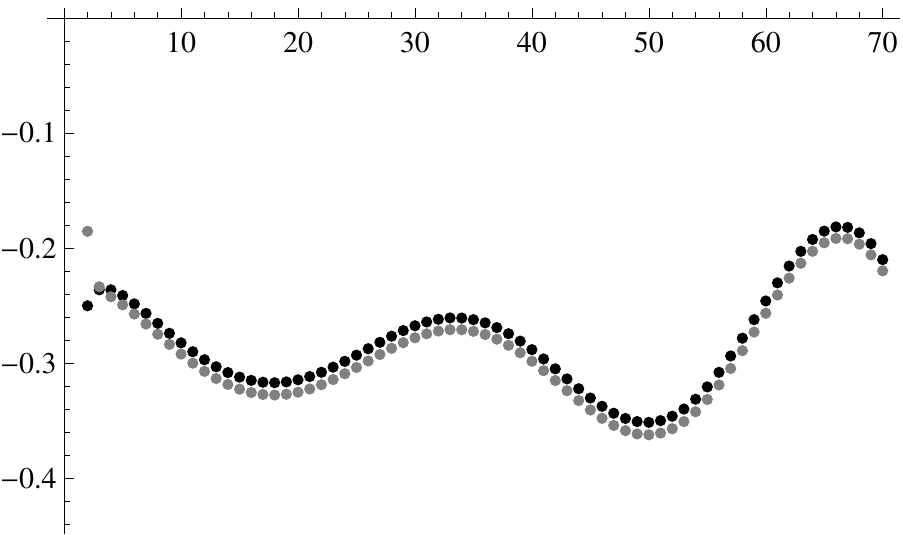}
\end{center}
\caption{\label{fig:approx}The numbers $C_{0,1,1}(N)$ (black)
and the approximation~\eqref{eq:new int} (gray), for
$N=1,\dots,70$.}
\end{figure}

\section{Saddle point asymptotics}\label{se:saddle}

We now proceed by a saddle point analysis of the integral~\eqref{eq:F},
using the approximation of the integrand provided by Lemma~\ref{le:mellin}.
According to this lemma, the factor
$ \exp\left(  \frac1z
  \left( \mathrm{Li}_2(e^z) - \frac{\pi^2}{6}\right) N \right)$
dominates the integrand in~\eqref{eq:F}. Equating its derivative to zero,
we obtain the saddle point~$z_0$ defined in~\eqref{eq:z0}. The argument
of its axis is (see~\cite{deBr58})
\begin{align*}
   a &= \frac{ \pi}{2} - \frac12 \arg \frac{d^2}{dz^2} \left.\left(
    \frac1z  \left( \mathrm{Li}_2(e^z) - \frac{\pi^2}{6}\right)\right)
    \right|_{z=z_0} \\
   &=  \frac{ \pi}{2} - \frac12 \arg\left( \frac{e^{z_0}}{z_0(1-e^{z_0})} \right) \approx 1.79,
\end{align*}
and $\rho=e^{ia}$ is thus the direction of steepest decent. By symmetry, the conjugate
$\bar{z}_0$ is a saddle point, too, and its direction of steepest descent is $\bar{\rho}$.
We now deform the integration
\begin{figure}
\begin{center}
\includegraphics[height=3in]{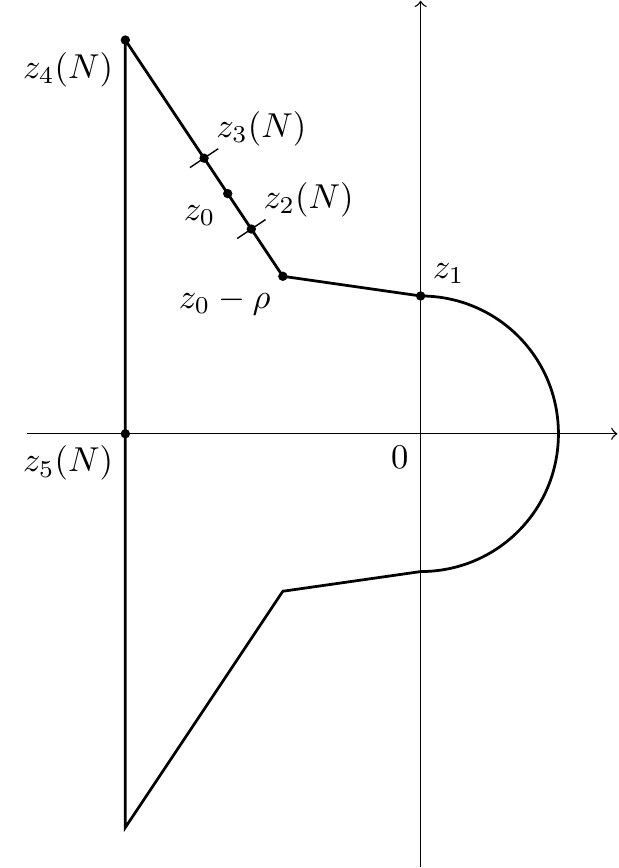}
\end{center}
\caption{\label{fig:path} The new integration contour, passing through
the saddle point~$z_0$. The (upper) dominating part lies between
$z_2(N)$ and $z_3(N)$. Angles and distances have been modified
for better visibility.}
\end{figure}
circle as follows (see Figure~\ref{fig:path}):
In the right half-plane, we stay with a half-circle, of radius~$5$. In the left half-plane,
we connect the point $z_1:=5i$ with the point $ z_0-\rho$ by a straight line. We then proceed by a
line through the saddle point~$z_0$,
up to a point $z_4(N)$. A vertical line then connects this point to the real
axis, to $z_5(N):=-\sqrt{N}$, and so $z_4$ must be
\[
  z_4(N) := -\sqrt{N} + i (\Im z_0-(\sqrt{N}+\Re z_0) \Im \rho / \Re \rho.
\]
Around the saddle point, we identify a range of width $O(N^{-39/112})$,
delimited by the points
\[
  z_2(N):=z_0-\rho N^{-39/112} \quad \text{and} \quad
  z_3(N):=z_0+\rho N^{-39/112}.
\]
In the lower half-plane, the contour is defined symmetrically.
We refer to the line from $z_2$ to $z_3$ to the (upper) central part
of the contour, as it gives the dominant contribution to the
integral (in the upper half-plane).
Note that $-39/112\approx -0.348$ is just a little bit smaller than $-1/3$.
To make the third-order term $N(z-z_0)^3$ in the local expansion of the integrand
negligible, we must have $z-z_0 \ll N^{-1/3}$. It is convenient
to make the central part as large as possible, though, because
this causes faster decrease of $F=e^f$ at $z=z_2(N)$
and $z=z_3(N)$, which
in turn makes it easier to beat the estimate for~$h$ from Lemma~\ref{le:mellin}.
(For details, see the tail estimate in Lemma~\ref{le:tail} below.)

Part~(i) of Lemma~\ref{le:mellin} provides the local
expansion in the central range, where $z=z_0+t\rho$,
$-N^{-39/112}\leq t \leq N^{-39/112}$:
\begin{multline}\label{eq:loc expans}
  f(z,N) = - N \log(1 -e^{z_0}) -\tfrac12 \alpha N t^2 - \tfrac12 \log N \\
     + (l-\tfrac12)\log(-z_0) -\tfrac12 (\log2\pi +\log(1-e^{z_0}))
      + O(N^{-5/112}).
\end{multline}
(Note that the expansion was simplified by using the defining equation~\eqref{eq:z0}
of~$z_0$.) The constant
\[
  \alpha := -\frac{\rho^2e^{z_0}} {z_0(1-e^{z_0})} \approx 0.028
\]
is real and positive. Since
\[
  \int_{-N^{-39/112}}^{N^{-39/112}}\exp(-\tfrac12 \alpha N t^2)dt
  \sim \frac{1}{\sqrt{\alpha N}}\int_{-\infty}^{\infty}e^{-x^2/2}dx
  = \sqrt{\frac{2\pi}{\alpha N}},
\]
with exponentially decaying error term,
the saddle point integral has the asymptotics
\begin{equation}\label{eq:sp int}
  \int_{z_2(N)}^{z_3(N)} e^{f(z,N)}dz =
  \frac{\rho(-z_0)^{l-1/2}}{ \sqrt{\alpha(1-e^{z_0})}}
  \frac1N (1-e^{z_0})^{-N}(1+O(N^{-5/112})).
\end{equation}
The integral over the lower saddle point segment is
\[
  \int_{\bar{z}_3(N)}^{\bar{z}_2(N)}e^{f(z,N)}dz =
    - \overline{\int_{z_2(N)}^{z_3(N)} e^{f(z,N)}dz},
\]
and thus the contribution of both saddle points to the integral~\eqref{eq:F}
is
\[
   \frac{(-1)^{l-1}}{\pi N^l} \Im \left(
     \int_{z_2(N)}^{z_3(N)} e^{f(z,N)}dz \right).
\]
By~\eqref{eq:sp int}, we see that this gives the right hand
side of~\eqref{eq:C asympt}.

To show that the two small line segments containing the saddle points~$z_0$
resp.~$\bar{z}_0$ capture
the asymptotics of the full integral~\eqref{eq:F},
we have to show that the remaining part of the contour in Figure~\ref{fig:path}
is negligible. 
By conjugation, it clearly suffices to consider
the upper half-plane.
We begin with the part that, additionally, lies
in the half-plane $\Re z \leq N^{-7/8}$. In the next section,
we show that the integral over the remaining part tends exponentially to zero.

\begin{lemma}\label{le:tail}
  \begin{multline}\label{eq:tail ints}
     \int_{z_1}^{z_2} \mathbf{1}_{\{\Re z \leq -N^{-7/8}\}}F(z,N) dz
     + \int_{z_3}^{z_4}F(z,N) dz \\
+ \int_{z_4}^{z_5}F(z,N)dz = O\Big(b^N \exp\Big( -\tfrac13 \alpha N^{17/56} \Big)\Big).
  \end{multline}
\end{lemma}
\begin{proof}
  We begin with the first integral. By part~(ii)
  of Lemma~\ref{le:mellin}, $f=\log F$ satisfies
  \[
    f(z,N) = \frac{1}{z}\Big( \mathrm{Li}_2(e^z) - \frac{\pi^2}{6}\Big)N
      + O(N^{33/112})
  \]
  there. It is straightforward to verify that the function
  \begin{equation}\label{eq:z Re}
   z\mapsto \Re\left(
  \frac{1}{z}\Big( \mathrm{Li}_2(e^z) - \frac{\pi^2}{6}\Big)\right)
  \end{equation}
  increases as $z$ moves along the contour from $z_1$ to $z_0$.
  By~\eqref{eq:loc expans}, we can therefore bound the first integral in~\eqref{eq:tail ints} by 
  \begin{align*}
    \exp &\left(    N \Re \left(     \frac{1}{z}
    \Big(     \mathrm{Li}_2(e^z) - \frac{\pi^2}{6}  \Big)
    \Big|_{z=z_2}
    \right) 
    + O(N^{33/112})  \right)    \\
    &= |1-e^{z_0}|^{-N} \exp\Big( -\tfrac12 \alpha N^{17/56} + O(N^{33/112})\Big) \\
    &=   O\Big(b^N \exp\Big( -\tfrac13 \alpha N^{17/56} \Big)\Big),
  \end{align*}
  where the length of the contour was absorbed into the $O$.
  The second integral in~\eqref{eq:tail ints} can be estimated analogously,
   by part~(i) of Lemma~\ref{le:mellin}.
  the function~\eqref{eq:z Re} decreases only eventually as $z$ moves
  from~$z_3$ to~$z_4$, but it is nowhere larger than at $z_3$, which suffices.

  Finally, we bound the last integral in~\eqref{eq:tail ints}. The function
  $h$ from~\eqref{eq:h} is $o(1)$ here, by part~(i) of Lemma~\ref{le:mellin}. 
  The factor $(e^z-1)^{-1/2}$ is $O(1)$, and $-\pi^2N/(6z)$ is
  $O(N^{1/2})$. The dilogarithm is $\mathrm{Li}_2(e^z) = O(e^{-\sqrt{N}})$,
  hence $(N/z)\mathrm{Li}_2(e^z) =  o(1)$,  and so
  \[
     f(z,N) = (l-\tfrac12)\log(-z) +O(N^{1/2}).
  \]
  As the integral of $(-z)^{l-1/2}$ from~$z_4$ to~$z_5$ grows only
  polynomially, the last integral in~\eqref{eq:tail ints}
  is $\exp(O(\sqrt{N}))$, and we are done.
\end{proof}

\section{Estimates close to the imaginary axis and in the right
half-plane}\label{se:right}

In the preceding section, have gave an asymptotic evaluation of the
integral~\eqref{eq:F}, where the contour was deformed as in 
Figure~\ref{fig:path}, and $\Re z < -N^{-7/8}$. We now show that
the remaining part of the contour is negligible.
Close to the imaginary axis, where $-N^{-7/8} \leq \Re z \leq 0$,
we are outside of the validity region of the Mellin transform asymptotics
of Lemma~\ref{le:mellin}. We thus estimate the integrand
in~\eqref{eq:F} directly.
\begin{lemma}\label{le:iR}
  We have
  \[
    \int_{z_1}^{z_2} \mathbf{1}_{\{\Re z \geq -N^{-7/8}\}}F(z,N) dz = O(0.85^N).
  \]
\end{lemma}
\begin{proof}
For simplicity, we assume that $z$ lies on a horizontal line,
so that $\Im z=5$; this is justified, because the monotonicity
used in the proof of Lemma~\ref{le:tail} persists if we adjust the contour like
this in a small neighborhood of~$5i$.
  It thus suffices to show that
  \begin{equation}\label{eq:prod est}
      \prod_{j=1}^N \frac{1}{|1-e^{z j/N}|} = O(0.84^N),
   \end{equation}
  uniformly for $\Im z=5$ and $-N^{-7/8} \leq \Re z \leq 0$,
  because all other factors in~\eqref{eq:F} grow subexponentially.
  A simple calculation yields
  \begin{equation}\label{eq:abs expl}
    \frac{1}{|1-e^{z j/N}|} = \left(1+  e^{2j \Re(z)/N} - 2\cos(5j/N)
    e^{j \Re(z)/N} \right)^{-1/2}.
  \end{equation}
  We divide the product~\eqref{eq:prod est}
  into $j\leq N/10$ and $j>N/10$. In the latter range,
  we have $1/(1-\cos(5j/N))=O(1)$, and thus
  \begin{align}
    1+  e^{2j \Re(z)/N} - 2\cos(5j/N) e^{j \Re(z)/N}
      &= 2(1-\cos(5j/N))+ O(N^{-7/8}) \notag \\
    &= 2(1-\cos(5j/N))(1+ O(N^{-7/8})). \label{eq:j big}
  \end{align}
  Now $(1+ O(N^{-7/8}))^N$ grows subexponentially and can be ignored (by rounding
  up the exponential factor we finally obtain slightly). The remaining product
  \begin{equation}\label{eq:prod sum}
    \prod_{N/10 \leq j\leq N}(1-\cos(5j/N))^{-1/2} =
    \exp\left(-\frac12 \sum_{N/10 \leq j\leq N} \log(1-\cos(5j/N)) \right)
  \end{equation}
  can be treated by Euler's summation formula.
  We have, with $\phi(x):=-\log(1-\cos(5x/N))$,
  \begin{equation}\label{eq:euler}
    \sum_{N/10 \leq j\leq N} \phi(j) = \int_{\lfloor N/10 \rfloor}^{N+1} \phi(x)dx
    -\frac12(\phi(N+1)-\phi(\lfloor N/10 \rfloor)) +
    \int_{\lfloor N/10 \rfloor}^{N+1}(\{x\}-\frac12)\phi'(x)dx.
  \end{equation}
  The term $-\tfrac12(\dots)$ is clearly $O(1)$.
  Since $\phi'(x) = -5/N \cot(5x/2N)$, the last integral can be estimated by
  \begin{align*}
    \int_{\lfloor N/10 \rfloor}^{N+1} | \phi'(x) | dx
    &= \int_{ N/10 }^{N+1} | \phi'(x) | dx  + O(1) \\
    &= -\int_{ N/10 }^{\pi N/5}\phi'(x) dx + \int_{\pi N/5}^{N+1}\phi'(x)dx +O(1) \\
    &= -2\phi(\pi N/5)+\phi(N/10)+\phi(N+1)+O(1) = O(1).
  \end{align*}
  The main integral in~\eqref{eq:euler} can be done in closed form
  (with Mathematica, e.g.):
  \[
    \int \phi(x) dx = -\frac{5i x^2}{2N} + 2x \log(1-e^{5ix/N})
    - x \log(1-\cos(5x/n)) - \frac25 i N \mathrm{Li}_2(e^{5ix/N}).
  \]
  From this we easily deduce
  \[
     \int_{\lfloor N/10 \rfloor}^{N+1} \phi(x)dx = -c N + O(1),
  \]
  where
  \begin{multline*}
    c = \frac{1}{40}(99 i +8\log(1-e^{i/2}) -80 \log(1-e^{5i})
    -4\log(1-\cos(1/2)) \\
     + 40 \log(1-\cos 5)     -16i \mathrm{Li}_2(e^{i/2})
    +16i \mathrm{Li}_2(e^{5i})) \approx 0.11262.
  \end{multline*}
  Inserting all this into~\eqref{eq:prod sum} yields
  \[
    \prod_{N/10 \leq j\leq N}(1-\cos(5j/N))^{-1/2} = \exp(-\tfrac12 cN + O(1)),
  \]
  and therefore (without forgetting the factor~$2$ in~\eqref{eq:j big})
  \begin{equation}\label{eq:prod j big}
    \prod_{N/10 \leq j\leq N} \frac{1}{|1-e^{z j/N}|} = O(2^{-9N/20} e^{-cN/2}
    (1+\varepsilon)^N ) = O(0.7^N).
  \end{equation}
  
  Now we treat the range $j\leq N/10$. For this we prove an appropriate
  inequality. Noting that $\cos(5j/N)$ is positive, and using truncated
  Taylor series three times, we obtain
  \begin{align*}
    1+  e^{2j \Re(z)/N}& - 2\cos(5j/N)  e^{j \Re(z)/N} 
    \geq 1+\left[1+x+\frac{x^2}{2}+\frac{x^3}{6} \right]_{x=2j \Re(z)/N} \\
    &- 2\left[1-\frac{x^2}{2}+\frac{x^4}{24} \right]_{x=5j/N} \cdot
    \left[1+x+\frac{x^2}{2} \right]_{x=j \Re(z)/N}.
  \end{align*}
  This can be bounded from below by
  \begin{equation}\label{eq:lower bd}
    \geq \frac{11j^2}{12N^2}(\Re(z)^2+25).
  \end{equation}
  The latter fact is a polynomial inequality with polynomial constraints,
  and can be established by cylindrical algebraic
  decomposition~\cite{Co75}, e.g., with Mathematica.
  Note that the form of~\eqref{eq:lower bd} was guessed from a Taylor expansion
  of~\eqref{eq:abs expl}
  for $j\approx 0$. From~\eqref{eq:abs expl}
  and~\eqref{eq:lower bd} we have the estimate
  \begin{align}
    \prod_{1\leq j\leq N/10} \frac{1}{|1-e^{z j/N}|} &\leq \prod_{1\leq j\leq N/10}
      \sqrt{\frac{12}{11}} \frac{N}{j} ((\Re z)^2 +25)^{-1/2} \notag \\
    &\leq \left( \frac{12}{275} \right)^{N/10} N^{N/10} \lfloor N/10 \rfloor!^{-1}(1+\varepsilon)^N  \notag \\
    & \leq \left(
     \left(\frac{12}{275}\right)^{1/20}(10e)^{1/10}(1+\varepsilon)\right)^N
    \leq 1.19^N, \label{eq:j small}
  \end{align}
  for $N$ large.
  Now multiply~\eqref{eq:prod j big} and~\eqref{eq:j small} to
  get the result.
\end{proof}

Finally, we estimate the integral over the right half-circle in~\eqref{eq:F},
which completes the proof of Theorem~\ref{thm:main}.
By the reflection formula~\eqref{eq:refl}, we can recycle part of the
analysis from the left half-plane.

\begin{lemma}
  \[
   \frac{(-1)^{l-1}}{N^l} \frac{1}{2i\pi} \int_{|z|=5}
     \mathbf{1}_{\{\Re z>0\}} F(z,N) dz = O(0.95^N).
  \]
\end{lemma}
\begin{proof}
  First consider the range $0< \Re z \leq N^{-7/8}$. All factors
  in front of the products in~\eqref{eq:F def} and~\eqref{eq:refl}
  grow at must subexponentially, and so this part of the integral
  is $O(0.85^N)$ by~\eqref{eq:refl} and~\eqref{eq:prod est}.
  
  On the other hand, for $N^{-7/8}< \Re z \leq 5$, the proof
  of Lemma~\ref{le:mellin} shows that the product in~\eqref{eq:refl}
  satisfies
  \[
    \prod_{j=1}^N \frac{1}{1-e^{-zj/N}} = \exp\left(
    -\frac{1}{z}\Big( \mathrm{Li}_2(e^{-z}) - \frac{\pi^2}{6}\Big)N
    +O(N^{33/112}) \right).
  \]
  The function
  \[
    z \mapsto \Re\left(-\frac{1}{z}
      \Big( \mathrm{Li}_2(e^{-z}) - \frac{\pi^2}{6}\Big) \right)
  \]
  increases as~$z$ moves on the arc from $-5i$ to $5$ and then
  decreases until $5i$. Close to the imaginary axis, where
  $|\arg z|\geq \tfrac{9}{20}\pi$, we thus
  have the bound
  \[
    \left|\exp\left(
    -\frac{1}{z}\Big( \mathrm{Li}_2(e^{-z}) - \frac{\pi^2}{6}\Big)N \right) \right|
    \leq 0.97^N,
  \]
  obtained by inserting $z=5\exp(\tfrac{9}{20}\pi i)$. Taking into account
  the subexponential factors, this portion of the integral
  is $O(0.98^N)$. If $|\arg z|\leq \tfrac{9}{20}\pi$, i.e., $z$ is bounded away
  from the imaginary axis, we get help from the factor $e^{-zN/2}$
  in~\eqref{eq:refl}. It is bounded by its absolute value
  at $z=5\exp(\tfrac{9}{20}\pi i)$, and thus not larger than $0.68^N$.
  Since
  \[
    \left|\exp\left(
    -\frac{1}{z}\Big( \mathrm{Li}_2(e^{-z}) - \frac{\pi^2}{6}\Big)N \right) \right|
    \leq 1.39^N,
  \]
  as found by plugging in $z=5$, the integral for $N^{-7/8}< \Re z$
  is $O(0.68^N \cdot 1.39^N)=O(0.95^N)$.
\end{proof}

\section{conclusion}

The error term that we obtained in~\eqref{eq:C asympt} can be improved a bit
by considering more terms of the local expansion~\eqref{eq:loc expans}
of~$f$ in the saddle
point analysis. Getting the \emph{correct} order of the error term, i.e.,
the next term in the asymptotic expansion of $C_{0,1,l}(N)$, needs some work, though.
As only the first term on the right-hand side of~\eqref{eq:f asympt}
was used to define the saddle point~$z_0$,
the logarithmic terms in~\eqref{eq:f asympt} contribute a non-vanishing
first order term $O(z-z_0)=O(N^{-39/112})$ to the expansion~\eqref{eq:loc expans}.
To improve it, we need to
replace~$z_0$ by a better approximation of the actual saddle point
of the whole integrand $F=e^f$. But then, the tail estimate in
Lemma~\ref{le:tail} becomes more involved,
because not only the width, but also the location of the saddle point
segment depends on~$N$.

Perhaps more importantly, we comment on possible future work.
Recall that Rademacher's conjecture essentially says that the operations
of limit and partial fraction decomposition commute in the present setting.
While our result refutes the conjecture, it does not clarify the relation
between the partial fraction decompositions of $\prod_{j\geq 1}(1-x^j)^{-1}$
and $\prod_{j=1}^N(1-x^j)^{-1}$; it would be surprising if there was none at all. O'Sullivan~\cite{oS12} suspects
that some modified version of the conjecture might hold, and presents
numerical evidence for convergence of $C_{h,k,l}(N)$ in terms
of Ces{\`ar}o means.

\bibliographystyle{siam}
\bibliography{../gerhold}

\end{document}